\documentclass[11pt]{article}
\usepackage[margin=1in]{geometry}
\usepackage[utf8]{inputenc}
\usepackage[T1]{fontenc}
\usepackage[english]{babel}
\usepackage{lmodern}
\usepackage{graphicx}
\usepackage{wrapfig}
\usepackage[subrefformat=parens,labelformat=parens]{subfig}
\usepackage{caption}
\usepackage{paralist}
\usepackage{enumitem}

\graphicspath{{./figures/}}

\makeatletter
\renewcommand{\p@enumii}{}
\makeatother

\newcommand{\titel}{A note on cycles in cyclically $4$-edge-connected cubic planar graphs}
\usepackage[usenames]{color}
\definecolor{hellblau}{rgb}{0.2,0.4,1} 
\definecolor{dunkelblau}{rgb}{0,0,0.8}
\definecolor{dunkelgruen}{rgb}{0,0.5,0}
\usepackage[
	pdftex,
	colorlinks,
	linkcolor=dunkelblau,
	urlcolor=dunkelblau,
	citecolor=dunkelgruen,
	bookmarks=true,
	linktocpage=true,
	pdfauthor={On-Hei Solomon Lo},
	pdfsubject={}
]{hyperref}
\usepackage{pdfpages}
\usepackage{amsmath}
\usepackage{amsthm}
\usepackage{amssymb}

\allowdisplaybreaks
\usepackage{amsfonts}
\theoremstyle{plain}
	\newtheorem{satz}{Satz}[]
	\newtheorem*{theorem}{Theorem}
	\newtheorem*{lemma}{Lemma}
	\newtheorem*{TEL}{The Three Edge Lemma}
	
	\newtheorem*{proposition}{Proposition}
		\newtheorem{claim}[satz]{Claim}
\theoremstyle{remark}
	
\theoremstyle{definition}
	
	\newtheorem*{conjecture}{Conjecture}
	
	\newtheorem*{bondyconjecture}{Bondy's conjecture}
	\newtheorem*{malkevitchconjecture}{Malkevitch's conjecture}

\newenvironment{proofblack}[1][Proof]{\par
  \pushQED{\qed}  
  \begin{proof}[#1] 
}{%
  \end{proof} 
}

\makeatletter
\newcommand{\setword}[2]{%
	\phantomsection
	#1\def\@currentlabel{\unexpanded{#1}}\label{#2}%
}
\makeatother

\begin{document}
	\title{\titel}
		\author{
		On-Hei Solomon Lo\thanks{School of Mathematical Sciences, 
			Key Laboratory of Intelligent Computing and Applications (Tongji University), Ministry of Education, Tongji University, Shanghai 200092, China. This research was supported by the Fundamental Research Funds for the Central Universities.
			}\\
		}
	\date{}
	\maketitle

\begin{abstract}
	Let $H$ be obtained from a cyclically $4$-edge-connected cubic planar graph $Y$ other than $K_4$ by deleting two adjacent vertices. We provide a short proof that if $H$ has circumference at least $k$ for some even integer $k \ge 4$, then $H$ contains a cycle of length between $k$ and $3k/2$.
	
As a consequence, we show that the line graph \( G \) of \( Y \) contains a cycle of length \( l \) avoiding any prescribed vertex of \( G \), for every \( l \in \{3\} \cup \{5, \dots, |V(G)| - 1\} \). 

The proofs integrate Euler's formula and the Three Edge Lemma, established by Thomas and Yu, and independently by Sanders, in a novel way. This work was partially motivated by conjectures of Bondy and Malkevitch.
\end{abstract}

\sloppy

\subsection*{Introduction}

A graph \( G \) with \( n \) vertices is called \emph{hamiltonian} if it contains a cycle that visits every vertex exactly once. Such a cycle is called a \emph{Hamilton cycle}. We say $G$ is \emph{pancyclic} if it contains cycles of all lengths from 3 to \( n \), and \emph{almost pancyclic} if it contains cycles of all such lengths except possibly one even length.

Whitney~\cite{Whitney1931} proved that a planar triangulation is hamiltonian if it is $4$-connected.
Tutte~\cite{Tutte1956} extended Whitney’s theorem by proving that every $4$-connected planar graph is hamiltonian. In fact, Tutte’s result implies that every vertex-deleted subgraph of a $4$-connected planar graph is hamiltonian (see~\cite{Plummer1975}).

Motivated by Bondy's metaconjecture that a sufficient condition for hamiltonicity almost always implies pancyclicity, one might ask whether $4$-connected planar graphs are pancyclic. However, Malkevitch~\cite{Malkevitch1971} showed that this is not always the case by proving that the line graph of a cyclically $4$-edge-connected cubic planar graph of girth 5 lacks a $4$-cycle. In light of this, Bondy~\cite{Bondy1975} in 1973 and Malkevitch~\cite{Malkevitch1988} in 1988 proposed the following conjectures.

\begin{bondyconjecture}
	Every $4$-connected planar graph is almost pancyclic.
\end{bondyconjecture}

\begin{malkevitchconjecture}
	Every $4$-connected planar graph is pancyclic if it contains a $4$-cycle.
\end{malkevitchconjecture}

It is known that every $4$-connected planar graph on \( n \) vertices contains an \( l \)-cycle for any \( l \in \{3, 5, 6\} \cup \{\lfloor \frac{n}{2} \rfloor, \dots, \lceil \frac{n}{2} \rceil + 3\} \cup \{n - 7, \dots, n\} \) \cite{Wang2002, Fijavz2002, Lo2021, Cui2009, Chen2004, Sanders1996, Thomas1994, Plummer1975, Thomassen1983, Tutte1956}. Weaker versions of Bondy's and Malkevitch's conjectures have been considered in \cite{Malkevitch1978, Chen2004}. However, we propose the following conjecture, which, if true, would establish both Bondy's and Malkevitch's conjectures.

\begin{conjecture}
	Let $G$ be a $4$-connected planar graph on $n$ vertices. For any $v \in V(G)$, there exists an $l$-cycle in $G - v$ for any $l \in \{3\} \cup \{5, \dots, n - 1\}$.
\end{conjecture}

As shown in \cite{Hakimi1979}, every planar triangulation contains cycles of all lengths up to its circumference. Thus every $4$-connected planar triangulation is pancyclic (see also \cite{Lo2025}) and, in fact, satisfies our conjecture (as vertex-deleted subgraphs of a $4$-connected planar graph are hamiltonian).

This note presents a short proof of our conjecture for the line graphs of cyclically $4$-edge-connected cubic planar graphs.

\begin{theorem}
	Let \( G \) be the line graph of a cyclically $4$-edge-connected cubic planar graph. Denote $n := |V(G)|$. For any \( v \in V(G) \), there exists an \( l \)-cycle in \( G - v \) for every \( l \in \{3\} \cup \{5, \dots, n - 1\} \).
\end{theorem}

To establish our theorem, we will investigate cycle lengths in cyclically $4$-edge-connected cubic planar graphs. A graph \( G \) is \emph{cyclically $4$-edge-connected} if it is $3$-connected and for every edge-cut \( S \) of \( G \) with fewer than 4 edges, \( G - S \) contains a component without cycles.

The \emph{length} of a cycle, denoted \( \ell(C) \), is its number of edges, and a cycle with \( k \) edges is a \emph{\( k \)-cycle}. In a 2-connected graph, the \emph{girth} is the length of a shortest cycle, and the \emph{circumference} is that of a longest cycle. A \emph{planar graph} can be embedded in the plane without edge crossings, while a \emph{plane graph} specifies a particular embedding. The regions left when a plane graph is removed are \emph{faces}, with the infinite one being the \emph{exterior face} and the rest \emph{interior faces}. The \emph{length} of a face is the number of edge-occurrences in the boundary walk of the face. In a 2-connected plane graph, every face is bounded by a \emph{facial cycle}.

\subsection*{Cycle lengths in a specified interval}

In this section we establish a proposition concerning cycle lengths in cyclically $4$-edge-connected cubic planar graphs, which is essential for proving our theorem.

\begin{lemma} \label{lem:acyclic}
	Let $D$ be a digraph without loops. There exists a spanning subdigraph $D'$ of $D$ such that $D'$ is acyclic (that is, contains no directed cycles) and $d_{D'}^+(v) + d_{D'}^-(v) \ge d_{D}^+(v)$ for all $v \in V(D')$.
\end{lemma}
\begin{proof}
	We prove by induction on $|V(D)|$. As it holds for $|V(D)| = 1$, we assume $|V(D)| > 1$. Let $v \in V(D)$ be such that $d_D^+(v) \le d_D^-(v)$. By the induction hypothesis, there is a spanning subdigraph $(D - v)'$ of $D - v$ such that $(D - v)'$ is acyclic and $d_{(D - v)'}^+(u) + d_{(D - v)'}^-(u) \ge d_{D - v}^+(u)$ for all $u \in V(D - v)$. Construct $D'$ from $(D - v)'$ by putting back the vertex $v$ and adding all edges from $E(D)$ that direct from some vertex in $V(D - v)$ to $v$. By the choice of $v$, $D'$ satisfies the required properties.
\end{proof}

\begin{proposition} \label{pro:k,3k/2strong}
	Let $Y$ be a cyclically $4$-edge-connected cubic planar graph with girth at least $4$. Let $H$ be obtained from $Y$ by deleting two adjacent vertices. Let $k \ge 4$ be an even integer. If $H$ has circumference at least $k$, then $H$ has some cycle of length in $[k, 3k / 2]$.
\end{proposition}
\begin{proof}
	Since $Y$ is a cyclically $4$-edge connected cubic planar graph with girth at least $4$, we have that $H$ is 2-connected. Moreover, $H$ has exactly four degree 2 vertices, which are contained in a facial cycle in $H$, denoted as $X$. We take an embedding of $H$ such that $X$ is the exterior facial cycle.
	
Suppose, for contradiction, that $H$ has no cycle length in $[k, 3k / 2]$. Since $Y$ has at least six faces of length 4 or 5 when $k = 4$, we assume $k \ge 6$. 

	Let $\mathcal{X}_1$ be the set of interior facial cycles of length in $[k / 2 + 4, k - 1]$ and $\mathcal{X}_2$ be the set of interior facial cycles of length at least $3k / 2 + 1$. Moreover, let $\mathcal{X} := \mathcal{X}_1 \cup \mathcal{X}_2$ be the set of interior facial cycles of length at least $k / 2 + 4$. 

    \begin{claim}
        $\mathcal{X} \neq \emptyset$.
    \end{claim}
    \begin{proofblack}
        Consider a cycle $C$ in $H$ such that $\ell(C) \ge k$ and, subject to this, the interior of $C$ contains a minimum number of faces. By our assumption, $C$ has length at least $3k / 2 + 1$. If $C$ contains some vertex $v$ in its interior, then there exist distinct $v_1, v_2, v_3 \in V(C)$ such that $v$ joins to $v_1, v_2, v_3$ with three internally disjoint paths as $Y$ is a $3$-connected planar graph. It is evident that the union of $C$ and these three paths contains a cycle of length at least $\frac{2}{3} \ell(C) > k$ and with fewer faces in its interior than $C$, which contradicts our choice. Hence, $C$ together with its interior induces an outerplanar graph. If $C$ has a chord $e$ lying in its interior, we denote by $C_1$ and $C_2$ the cycles containing $e$ in $C + e$. Indeed, $e$ can be chosen such that $C_1$ is a facial cycle. By the choice of $C$, $\ell(C_2) \le k - 1$. Since $(\ell(C_1) - 1) + (\ell(C_2) - 1) = \ell(C) \ge \frac{3k}{2} + 1$, we have $\ell(C_1) \ge \frac{3k}{2} + 3 - \ell(C_2) \ge \frac{k}{2} + 4$. Hence, $C_1$ is a facial cycle of length at least $\frac{k}{2} + 4$. If $C$ has no chord in its interior, then $C$ is the desired facial cycle of length at least $\frac{3k}{2} + 1 > \frac{k}{2} + 4$. 
        \end{proofblack}

	As $Y$ is a $3$-connected cubic planar graph, the intersection of any two interior facial cycles of $H$ is either empty or induced by a single edge. For any facial cycle $F$ of $H$ and any edge $e$ in $F$, denote by $F^e$ the facial cycle containing $e$ other than $F$ in $H$. Since $Y$ is cyclically $4$-edge-connected, for any interior facial cycle $F$ of $H$ and any independent edges $e_1$ and $e_2$ of $F$ that are not contained in $X$, the facial cycles $F^{e_1}$ and $F^{e_2}$ are also interior facial cycles and hence disjoint. For the same reason, any interior facial cycle of $H$ can contain at most two edges from $X$. 
	
\begin{claim}
    For each $F \in \mathcal{X}$, $F$ contains four independent edges $e$ with $F^{e} \in \mathcal{X} \cup \{X\}$.
\end{claim}
\begin{proofblack}
    We first consider $F \in \mathcal{X}_1$. Notice that we have $k \ge 10$ because in this case $\mathcal{X}_1 \neq \emptyset$. Since $F$ has length $\ell(F) \ge k / 2 + 4$, it has $\lfloor (k / 2 + 4) / 2 \rfloor = \lfloor k / 4 \rfloor + 2 \ge 4$ independent edges. We claim that at least four of these independent edges $e$ satisfy $F^{e} \in \mathcal{X} \cup \{X\}$. Otherwise, we have $t_1 := \lfloor k / 4 \rfloor - 1$ independent edges $e_1, \dots, e_{t_1}$ of $F$ such that none of $F^{e_1}, \dots, F^{e_{t_1}}$ belongs to $\mathcal{X} \cup \{X\}$. So, $F^{e_1}, \dots, F^{e_{t_1}}$ are pairwise disjoint interior facial cycles each satisfying $4 \le \ell(F^{e_i}) \le k / 2 + 3$. For $i \in \{0, 1, \dots, t_1\}$, define $C_i$ to be the graph obtained from the multigraph $F \cup F^{e_1} \cup \dots \cup F^{e_i}$ by deleting all edges that appear more than once. It is obvious that every $C_i$ is a cycle. For $i \in \{0, 1, \dots, t_1 - 1\}$, if $\ell(C_i) \le k - 1$, then $\ell(C_{i + 1}) = \ell(C_i) + \ell(F^{e_{i + 1}}) - 2 \le 3k / 2$ and hence $\ell(C_{i + 1}) \le k - 1$. As $\ell(C_0) = \ell(F) \le k - 1$, we have that \begin{align*}
		\ell(C_{t_1}) \le k - 1.
	\end{align*} On the other hand, $\ell(C_{i + 1}) \ge \ell(C_i) + 2$ as $\ell(F^{e_{i + 1}}) \ge 4$. This leads to the contradiction that \begin{align*}
	\ell(C_{t_1}) &\ge \ell(F) + 2 t_1
	> \left(\frac{k}{2} + 4\right) + 2 \left(\frac{k}{4} - 2\right) = k.
\end{align*}

Now, we consider $F \in \mathcal{X}_2$. If $F$ does not have four independent edges $e$ with $F^{e} \in \mathcal{X} \cup \{X\}$, then, as $\ell(F) \ge \frac{3k}{2} + 1 \ge 10$, there exists some edge $e$ in $F$ with $F^e \notin \mathcal{X} \cup \{X\}$. Hence, there are at most six edges $e$ with $F^{e} \in \mathcal{X} \cup \{X\}$. Since the edges $e$ in $F$ with $F^e \notin \mathcal{X} \cup \{X\}$ induce at most three (disjoint) paths, there are $t_2 := \lceil(({3k}/{2} + 1) - 6) / 3\rceil = \frac{k}{2} - 1$ consecutive edges $e_1, \dots, e_{t_2}$ in $F$ such that $F^{e_i} \notin \mathcal{X} \cup \{X\}$ for every $i \in \{1, \dots, t_2\}$. We have $4 \le \ell(F^{e_i}) \le \frac{k}{2} + 3$. For $i \in \{1, \dots, t_2\}$, define $C_i$ to be the cycle obtained from the multigraph $F^{e_1} \cup \dots \cup F^{e_i}$ by deleting all edges that appear more than once. Similarly to the previous case, one can easily prove that for $i \in \{1, \dots, t_2 - 1\}$, $\ell(C_{i + 1}) \le k - 1$ whenever $\ell(C_{i}) \le k - 1$. Since $\ell(C_1) \le k - 1$, we have \begin{align*}
\ell(C_{t_2}) \le k - 1.
\end{align*} On the other hand, we have the following contradiction \begin{align*}
\ell(C_{t_2}) &\ge 2 + \sum_{i = 1}^{t_2} (\ell(F^{e_i}) - 2) \ge 2 + 2t_2 = k. \qedhere
\end{align*}
\end{proofblack}

We define a planar digraph $\mathcal{D}$ with vertex set $\mathcal{X} \cup E(X)$ as follows. For each $F \in \mathcal{X}$, choose four independent edges $e$ satisfying $F^e \in \mathcal{X} \cup \{X\}$ and for each such $e$, assign an edge directed from $F$ to $F^e$ if $F^e \in \mathcal{X}$ and to $e$ if $F^e = X$. We have $d_\mathcal{D}^+(F) = 4$ for any $F \in \mathcal{X}$, while $d_\mathcal{D}^+(e) = 0$ and $d_\mathcal{D}^-(e) \le 1$ for any $e \in E(X)$. Let $\mathcal{D}'$ be the spanning acyclic subdigraph of $\mathcal{D}$ assured by the Lemma. Let $\mathcal{U}$ be the underlying undirected planar graph of $\mathcal{D}'$. Hence $d_\mathcal{U}(F) \ge d_\mathcal{D}^+(F) = 4$ for any $F \in \mathcal{X}$ and $d_\mathcal{U}(e) \le 1$ for any $e \in E(X)$. Let $\mathcal{K}$ be any component of $\mathcal{U}$ that contains some vertex from $\mathcal{X}$. 

\begin{claim}
    $\mathcal{K}$ contains some face of length $3$.
\end{claim}

\begin{proofblack}
    Suppose to the contrary that all faces of $\mathcal{K}$ have length at least 4. Denote by $f_{\mathcal{K}}$ the number of faces of $\mathcal{K}$. 

If $|V(\mathcal{K}) \cap E(X)| \le 2$, then $2|E(\mathcal{K})| = \sum_{v \in V(\mathcal{K})} d_\mathcal{U}(v) \ge 4(|V(\mathcal{K})|-2) + 2$. By Euler's formula, we have
\begin{align*}
8 = 4|V(\mathcal{K})| - 4|E(\mathcal{K})| + 4 f_{\mathcal{K}} \le (2|E(\mathcal{K})| + 6) - 4|E(\mathcal{K})| + 2|E(\mathcal{K})| = 6,
\end{align*} which is absurd. 

Thus we have $|V(\mathcal{K}) \cap E(X)| \ge 3$. 

	Recall that $Y$ is a cyclically $4$-edge-connected cubic planar graph with girth at least 4. Among the four faces of $Y$ that are not in $H$, precisely two of them have disjoint facial cycles, which we denote by $D_1$ and $D_2$. In $H$, among the paths obtained from $X$ by removing the degree 2 vertices, let $P_1$ and $P_2$ be the paths contained in $D_1$ and $D_2$, respectively. If any interior facial cycle of $H$ contains two independent edges from $X$, then one of these edges must be contained in $P_1$ and the other in $P_2$.

Let $\mathcal{K}'$ be obtained from two disjoint copies of $\mathcal{K}$ by, for each vertex in $V(\mathcal{K}) \cap E(X)$, identifying the corresponding pair of vertices. The identified vertices are regarded as elements from $V(\mathcal{K}) \cap E(X)$. Define $\mathcal{Q}$ from $\mathcal{K}'$ by adding edges that form the cycle $\mathcal{O}$ on $V(\mathcal{K}) \cap E(X)$, induced by the cyclic order in which the elements of $V(\mathcal{K}) \cap E(X)$ occur in $X$. So $\mathcal{Q}$ is planar and has minimum degree at least 4. Again, it follows easily from Euler's formula that $\mathcal{Q}$ has at least eight facial cycles of length 3. Let $\mathcal{S}$ be a facial cycle of length 3 not contained in any of the two copies of $\mathcal{K}$. Then $\mathcal{S}$ must contain an edge of $\mathcal{O}$ and some vertex $F$ not in $\mathcal{O}$, and $F$ must correspond to an interior facial cycle in $H$ that contains one edge from $P_1$ and one edge from $P_2$. Thus, by planarity, in each copy of $\mathcal{K}$ there are at most two such vertices $F$, and hence $\mathcal{Q}$ has at most four facial cycles of length 3 not contained in any copy of $\mathcal{K}$. In particular, some copy of $\mathcal{K}$ contains a facial cycle of length 3. The claim thus follows.
\end{proofblack}

Let $\mathcal{S}$ be a facial cycle of length 3 in $\mathcal{U}$, as guaranteed by the preceding claim. We have $V(\mathcal{S}) \subseteq \mathcal{X}$ because $d_\mathcal{U}(e) \le 1$ for all $e \in E(X)$. Let $F \in \mathcal{X}$ be the vertex of $\mathcal{S}$ such that $d^+_{\mathcal{D}'[V(\mathcal{S})]}(F) = 2$, which exists as $\mathcal{D}'$ is acyclic. Therefore, $F$ has two independent edges $e_1$ and $e_2$ such that $F^{e_1}$ and $F^{e_2}$ are interior facial cycles in $H$ and intersect. However, this is impossible since $Y$ is a cyclically $4$-edge-connected cubic planar graph.
\end{proof}

The above proof extends to show that any cyclically $4$-edge-connected cubic planar graph with circumference at least \( k \), with \( k \geq 4 \) even, has a cycle length in \( [k, \frac{3k}{2}] \), though the tightness of this interval remains unknown. For an analogous study on other classes of $3$-connected planar graphs, we refer to~\cite{Cui2021, Lo2024}.

\subsection*{Cycle lengths in the line graph}

The following powerful tool, instrumental in finding long cycles in planar graphs, was established by Thomas and Yu~\cite{Thomas1994} and, independently, Sanders~\cite{Sanders1996}.

\begin{TEL} \label{lem:TEL}
	Let $H$ be a $2$-connected plane graph. Let $X$ be a facial cycle of $H$, and let $t, y, s$ be edges of $X$. There exists a cycle $C$ in $H$ containing $t, y, s$ such that every component of $H - V(C)$ joins to at most three vertices of $C$, and every component of $H - V(C)$ that contains some vertex of $X$ joins to at most two vertices of $C$.
\end{TEL}

We are now ready to prove  the Theorem. 

\begin{proof}[Proof of  the Theorem]
Let \( G \) be the line graph of a cyclically $4$-edge-connected cubic planar graph \( Y \), and $v$ be a vertex of $G$. The case where \( Y \cong K_4 \) trivially holds, so we assume \( Y \) has girth at least 4. We have \( n = |V(G)| \geq 12 \).

We divide the proof into three parts, establishing the existence of cycles of lengths $3$ and $5$, of lengths in $[6, \frac{2n}{3})$, and of lengths in $[\frac{2n}{3}, n]$ in $G - v$.

\smallskip

We first show that $G - v$ has an $l$-cycle for every $l \in \{3, 5\}$. The existence of a $3$-cycle in $G-v$ follows from the fact that $G$ has two disjoint facial $3$-cycles. Since $Y$ has two disjoint facial cycles each of length at most 5, $G$ must contain a facial cycle $Q$ of length 4 or 5 that does not contain $v$. If $Q$ has length 4, consider a facial $3$-cycle that intersects $Q$ but does not contain $v$. The union of $Q$ and this facial $3$-cycle clearly contains a 5-cycle. This proves the existence of a 5-cycle.

\smallskip

    Next, we prove the existence of the large cycle lengths. The arguments in this part were given in~\cite{Lo2021a}; we include them here for completeness.

	Recall that every vertex of $G$ has degree 4 and $E(G)$ can be partitioned into precisely $\frac{2n}{3}$ edge-disjoint facial $3$-cycles. So every vertex is contained in two of these cycles. Let $e_1, e_2$ be two edges incident with $v$ such that they induce a $3$-cycle with another edge $e_3$. 
	
We apply the Three Edge Lemma by taking \( H := G - e_3 \), \( X \) as the facial cycle of \( G - e_3 \) containing \( e_1 \) and \( e_2 \), and the three edges as \( e_1, e_2 \), and any other edge in \( X \). Let \( C \) be the guaranteed cycle. Suppose \( C \) is not a Hamilton cycle in \( G - e_3 \). Then there exists a component \( A \) of \( H - V(C) \). Since \( G \) is $4$-connected and \( \ell(C) > 3 \), \( A \) must join to at least four vertices of \( C \), contradicting the Three Edge Lemma. Thus $C$ is a Hamilton cycle in $G-e_3$ containing $e_1$ and $e_2$.

	We consider $C$ as a Hamilton cycle in $G$. For $j \in \{0, 1, 2\}$, a \emph{$j$-triangle} is one of the $\frac{2n}{3}$ facial $3$-cycles that contains precisely $j$ edges of $C$. It is clear that every facial $3$-cycle of $G$ is a 0-, 1-, or 2-triangle. The \emph{center} of a 2-triangle is the vertex incident to the two edges that belong to both $C$ and the 2-triangle. For instance, $v$ is the center of a 2-triangle. We can obtain a cycle whose length is one less than the original cycle by removing the center of a 2-triangle and adding the edge of the 2-triangle that is not in the original cycle. Denote by $\tau_j$ the number of $j$-triangles.    As $C$ is a Hamilton cycle, we have $n = |V(C)| = 2\tau_2 + \tau_1$. Observe that every 0-triangle contains three centers of 2-triangles and every 1-triangle contains one center. By counting the centers of 2-triangles, we obtain $\tau_2 = 3\tau_0 + \tau_1 \ge \tau_1$. It follows that $\tau_2 \ge \frac{n}{3}$, from which we conclude that $G - v$ has an $l$-cycle for every $l \in \{\frac{2n}{3}, \dots, n - 1\}$. 
	
\smallskip
    
    It remains to establish the cycle lengths in $[6, \frac{2n}{3})$. Suppose $l \in [6, \frac{2n}{3})$. Write $l = 3a + r$ with $r \in \{0, 1, 2\}$. Let $H$ be obtained from $Y$ by deleting the end-vertices of the edge that corresponds to $v$. We claim that $H$ has circumference at least $2a$. Denote by $v_1, v_2, v_3, v_4$ the degree 2 vertices of $H$ and by $X$ the facial cycle containing these vertices. Let $e_i$ be an edge incident with $v_i$ in $X$. Let $C$ be a cycle containing $e_1, e_2, e_3$ as assured by the Three Edge Lemma. Consequently, $C$ contains $v_1, v_2, v_3$. As $Y$ is cyclically $4$-edge-connected, every component of $H - V(C)$ is an isolated vertex, and it joins to only two vertices of $C$ if and only if it is comprised of the vertex $v_4$. We have $\ell(C) \ge 3 + 3 \cdot (|V(H) \setminus V(C)| - 1) + 2 \cdot 1 = 3 |V(H) \setminus V(C)| + 2$. Therefore
	\begin{align*}
		\ell(C) &\ge \frac34|V(H)| + \frac12 = \frac34\left(\frac{2n}{3} - 2\right) + \frac12 = \frac{n}{2} - 1 \ge \frac23\left( \frac{2n}{3} - 1\right) \ge \frac23 l \ge 2a.
	\end{align*}

	Now, as $H$ has circumference at least $2a$, we apply the Proposition (with $k := 2a$) to conclude that $H$ has some cycle $C_H$ of length in $[2a, 3a]$. We consider $C_H$ as a cycle in $Y$.

As the edges of $Y$ correspond bijectively to the vertices of $G$, we do not distinguish between an edge of $Y$ and its corresponding vertex in $G$. Under this identification, the cyclic order of $E(C_H)$ along $C_H$ induces a cycle $C_G$ in $G$ on the vertex set $E(C_H) \subseteq V(G)$. 

It follows that $2a\le\ell(C_G) = \ell(C_H)\le 3a$. 

Since $C_H$ is a cycle in $H$, $C_G$ does not intersect any facial $3$-cycle containing $v$.

Every edge $e$ of $C_G$ lies in a unique facial $3$-cycle of $G$ whose intersection with $C_G$ consists precisely of $e$ and its end-vertices; let $s_e$ denote the vertex of this facial $3$-cycle other than the end-vertices of $e$. For any $e \in E(C_G)$, we have $s_e \ne v$, because $C_G$ does not meet any facial $3$-cycle containing $v$. If two distinct edges $e_1, e_2 \in E(C_G)$ satisfy $s_{e_1} = s_{e_2} = s$, then $s$, viewed as an edge of $Y$, is a chord of $C_H$. Conversely, for any chord $s$ of $C_H$ in $Y$, exactly two edges $e \in E(C_G)$ satisfy $s_e = s$.

Let $\Lambda$ be a maximal subset of $E(C_G)$ such that $s_{e_1} \ne s_{e_2}$ for all distinct $e_1, e_2 \in \Lambda$. Since $v_1, v_2, v_3$ lie on $C_H$ and are not incident with any chord of $C_H$ in $Y$, at least three edges $e \in E(C_G)$ have $s_e$ not identifying to a chord of $C_H$. Together with the observations above, this implies that $|\Lambda| - 3 \ge \ell(C_G) - |\Lambda|$. Thus $|\Lambda| \ge \left\lceil \tfrac{\ell(C_G)+3}{2}\right\rceil$.

For any subset $\Lambda' \subseteq \Lambda$, we can modify $C_G$ via $\Lambda'$ as follows. For each $e \in \Lambda'$, delete $e$ and add the vertex $s_e$ along with the two edges joining $s_e$ to the end-vertices of $e$. This produces a cycle on $V(C_G) \cup \{s_e : e \in \Lambda'\}$. Its length is $\ell(C_G) + |\Lambda'|$, and it avoids $v$. 

Varying $\Lambda'$ over the subsets of $\Lambda$ shows that $G-v$ contains cycles of every length from $\ell(C_G)$ to $\ell(C_G) + |\Lambda|$. Since
\begin{align*}
\ell(C_G)&\le 3a \le l \le 3a+2 = 2a + \left\lceil \tfrac{2a+3}{2}\right\rceil \le \ell(C_G) + \left\lceil \tfrac{\ell(C_G)+3}{2}\right\rceil \le \ell(C_G) + |\Lambda|,
\end{align*}
this establishes the existence of an $l$-cycle in $G-v$.
\end{proof}

\section*{Acknowledgments}
The author thanks the two anonymous referees for their careful reading of the manuscript and for their invaluable comments, which significantly improved the presentation.

\bibliographystyle{abbrv}
\bibliography{paper}

\end{document}